\documentclass[a4paper]{amsart}

\newtheorem{theorem}{Theorem}[section]
\newtheorem{lemma}[theorem]{Lemma}
\newtheorem{corollary}[theorem]{Corollary}

\theoremstyle{remark}
\newtheorem{remark}[theorem]{Remark}

\numberwithin{equation}{section}

\newcommand{\indic}[1]{1_{{#1}}}  
\DeclareMathOperator*{\EE}{\mathbb{E}} 
\newcommand*{\PP}{\mathbb{P}} 
\newcommand*{\RR}{\mathbb{R}} 
\newcommand*{\NN}{\mathbb{N}} 
\newcommand*{\eps}{\varepsilon} 
\DeclareMathOperator*{\sgn}{sgn} 

\newcommand*{\IC}[1][\beta]{\text{IC}(#1)} 
\newcommand*{\infconv}{\square} 
\newcommand*{\calL}{\mathcal{L}} 

\newcommand*{\krok}[1]{\emph{#1}} 

\begin{document}

\title[Convex infimum convolution inequality]{On the convex infimum convolution inequality with optimal cost  function}

\author[M. Strzelecka]{Marta Strzelecka}
\address{Institute of Mathematics, University of Warsaw, Banacha 2, 02--097 Warsaw, Poland.}
\email{martast@mimuw.edu.pl}

\author[M. Strzelecki]{Micha{\l} Strzelecki}
\address{Institute of Mathematics, University of Warsaw, Banacha 2, 02--097 Warsaw, Poland.}
\email{michalst@mimuw.edu.pl}

\author[T. Tkocz]{Tomasz Tkocz}
\address{Mathematics Department, Princeton University, Fine Hall, 
Princeton, NJ 08544-1000
USA.}
\email{ttkocz@princeton.edu}

\thanks{Research partially
supported by the National Science Centre, Poland, grants no. 2015/19/N/ST1/02661 (M. Strzelecka) and  2015/19/N/ST1/00891 (M. Strzelecki) as well as the Simons Foundation (T. Tkocz)}


\date{February 23, 2017}


\begin{abstract}
We show that every symmetric random variable with log-concave tails satisfies the convex  infimum convolution inequality with an optimal cost function (up to scaling). As a result, we obtain nearly optimal comparison of weak and strong moments for symmetric random vectors with independent coordinates with log-concave tails.
\end{abstract}

\maketitle


\section{Introduction}

Functional inequalities such as the Poincar\'e, log-Sobolev, or Marton-Talagrand inequality to name a few, play a crucial role in studying concentration of measure, an important cornerstone of the local theory of Banach spaces.  
In this paper we focus on another example of such inequalities, the infimum convolution inequality, introduced by Maurey in \cite{MR1097258}.

Let $X$ be a random vector with values in $\RR^n$ and let $\varphi:\RR^n\to[0,\infty]$ be a~measurable function. We say that the pair $(X, \varphi)$ satisfies the \emph{infimum convolution inequality} (ICI  for short) if for every bounded measurable  function $f:\RR^n\to\RR$,
\begin{equation}\label{eq:def-IC}
\EE e^{f\infconv \varphi (X)}\EE e^{-f(X)}\leq 1,
\end{equation}
where $f \infconv \varphi$ denotes the infimum convolution of $f$ and $\varphi$ defined as $f \infconv \varphi(x) = \inf\{f(y) + \varphi(x-y) : y\in\RR^n\}$ for $x\in\RR^n$. The function $\varphi$ is called a \emph{cost function} and $f$ is called a \emph{test function}. We also say that the pair $(X, \varphi)$ satisfies the \emph{convex  infimum convolution inequality} if~\eqref{eq:def-IC} holds for every convex function $f:\RR^n\to\RR$ bounded from below.

Maurey showed that Gaussian and exponential random variables satisfy the ICI with a quadratic and quadratic-linear cost function  respectively. Thanks to the tensorisation property of the ICI, he recovered the Gaussian concentration inequality as well as the so-called Talagrand two-level concentration inequality for the exponential product measure. Moreover, Maurey proved that bounded random variables satisfy the convex ICI with a quadratic cost function  (see also Lemma 3.2 in \cite{ss17} for an improvement). 

Later on, Maurey's idea was developed further by Lata\l a and Wojtaszczyk who studied comprehensively the ICI in \cite{MR2449135}. By testing with linear functions, they observed that the optimal cost function is given by the Legendre transform of the cumulant-generating function (here optimal means largest possible, up to a~scaling constant, because  the larger the cost function is, the better \eqref{eq:def-IC} gets). They introduced the notion of optimal infimum convolution inequalities, established them for log-concave product measures and uniform measures on $\ell_p$-balls, and put forward important, challenging and far-reaching conjectures (see also \cite{Latala-IMAbookfinal}).

The recent works \cite{grst14} and \cite{grsst15} enable to view the ICI from a different perspective. In \cite{grst14} the authors introduce weak transport-entropy inequalities and establish their dual formulations.
The dual formulations are exactly the convex ICIs. In \cite{grsst15} the authors investigate extensively the weak transport cost inequalities on the real line, obtaining a characterisation for arbitrary cost functions which are convex and quadratic near zero, thus providing a tool for studying the convex ICI. Around the same time, the convex ICI for the quadratic-linear cost function was fully understood in \cite{fmnw15}.

In this paper we go along Lata\l a and Wojtaszczyk's line of research and study the optimal convex ICI. Using the aforementioned novel tools from \cite{grsst15}, we show that product measures with symmetric marginals having log-concave tails satisfy the optimal convex ICI, which complements Lata\l a and Wojtaszczyk's result about log-concave product measures. This has applications to concentration  and moment comparison of any norm of such vectors in the spirit of celebrated Paouris' inequality (see \cite{MR2276533} and \cite{MR3150710}) and addresses some questions posed lately in \cite{lat-strz-2}. 
We also offer an example showing that the assumption of log-concave tails cannot be weakened substantially.

\section{Main results}

For a random vector $X$ in $\RR^n$ we define
\begin{equation*}
\Lambda^*_X(x) := \mathcal{L} \Lambda_X (x) := \sup_{y\in\RR^n} \{\langle x,y\rangle - \ln \EE  e^{\langle y,X\rangle}\},
\end{equation*}
which is the Legendre transform of  the cumulant-generating function 
\[
\Lambda_X(x) :=\ln \EE e^{\langle x, X\rangle},\qquad x\in\RR^n.\]

If $X$ is symmetric and the pair $(X, \varphi)$ satisfies the ICI, then $\varphi(x) \leq \Lambda^*_X(x)$ for every $x\in\RR^n$ (see Remark~2.12 in~\cite{MR2449135}). In other words, $\Lambda^*_X$ is the optimal cost function $\varphi$ for which the ICI can hold. Since this conclusion is obtained by testing~\eqref{eq:def-IC} with linear functions, the same holds for the convex ICI. Following \cite{MR2449135} we shall say that $X$ satisfies (convex) $\IC$ if the pair $(X, \Lambda_X^*(\cdot/\beta))$ satisfies the (convex) ICI.

We are ready to present our first main result.

\begin{theorem}\label{thm:main}
Let $X$ be a symmetric random variable with log-concave tails, i.e. such that the function
\begin{equation*}
t\mapsto N(t):= - \ln \PP (|X|\ge t), \quad  t\geq 0,
\end{equation*}
is convex. 
Then there exists a universal constant $\beta\leq 1680e$ such that $X$ satisfies convex $\IC$.
\end{theorem}

The (convex) ICI tensorises and, consequently, the property (convex) $\textrm{IC}$ tensorises: if independent random vectors $X_i$ satisfy (convex) $\textrm{IC}(\beta_i)$, $i=1,\ldots,n$, then the vector $(X_1,\ldots,X_n)$ satisfies (convex) $\textrm{IC}(\max \beta_i)$ (see \cite{MR1097258} and \cite{MR2449135}). Therefore we have the following corollary.

\begin{corollary}
Let $X$ be a symmetric random vector with values in $\RR^n$ and independent coordinates with log-concave tails. Then $X$ satisfies convex $\IC$  with a~universal constant $\beta\leq 1680e$.
\end{corollary}

Note that the class of distributions from Theorem~\ref{thm:main} is wider than the class of symmetric log-concave product distributions considered by Lata{\l}a and Wojtaszczyk in \cite{MR2449135}. Among others, it contains measures which do not have a connected support, e.g. a symmetric Bernoulli random variable.

In order to comment on the relevance of the assumptions of Theorem~\ref{thm:main}  and present applications to comparison of weak and strong moments, we need the following definition.
	Let $X$ be a random vector with values in $\RR^n$. We say that the moments of $X$ \emph{grow $\alpha$-regularly} if for every $p\geq q\geq 2$ and every $\theta \in S^{n-1}$ we have
	\begin{equation*}
		\| \langle X,\theta \rangle \|_p \le \alpha \frac{p}q \| \langle X,\theta \rangle \|_q,
	\end{equation*}
	where $\|Y\|_p:=  (\EE |Y|^p)^{1/p}$ is the $p$-th integral norm of a random variable $Y$.
Clearly, if the moments of $X$ grow $\alpha$-regularly, then  $\alpha$ has to be at least $1$ (unless $X=0$ a.s.). 

\begin{remark}\label{rem:tails-reg}
If $X$ is a symmetric random variable with log-concave tails, then its moments grow $1$-regularly (this classical fact follows for instance from Proposition 5.5 from \cite{MR3329056} and the proof of Proposition 3.8 from \cite{MR2449135}). 
\end{remark}

The assumption of log-concave tails in Theorem \ref{thm:main} cannot be replaced by a~weaker one of $\alpha$-regularity of moments: if $X$ is a symmetric random variable defined by
\begin{equation}
\label{eq:example-intro}
\PP(|X| >t ) = \indic{[0,2)} (t) + \sum_{k=1}^{\infty} e^{-2^k} \indic{[2^k, 2^{k+1})}(t), \quad t\geq 0,
\end{equation}
then the moments of $X$ grow $\alpha$-regularly (for some $\alpha<\infty$), but there does not exists $C>0$ such that the pair $(X, x\mapsto\max\{(Cx)^2,C|x|\})$ satisfies the convex ICI. All the more, $X$ cannot satisfy convex $\IC$ with any $\beta<\infty$ (see Section~\ref{sec:example} for details). Thus it seems that the assumptions of Theorem~\ref{thm:main} are not far from necessary  conditions for the convex ICI to hold with an optimal cost function (random variables with moments growing regularly are akin to random variables with log-concave tails 
as the former can essentially be sandwiched between the latter,
see (4.6) in \cite{MR3335827}). 

Our second main result is an application of Theorem~\ref{thm:main} to moment comparison.  Recall that for a~random vector $X$ its $p$-th weak moment associated with a norm $\|\cdot\|$ is the quantity defined as
	\begin{equation*}
		\sigma_{\|\cdot\|, X}(p) := \sup_{\|t\|_* \le 1} \|\langle t,X\rangle \|_p,
	\end{equation*}
	where $\|\cdot\|_*$ is the dual norm of $\| \cdot\|$.
The following version of \cite[Proposition~3.15]{MR2449135}
holds (some non-trivial modifications of the proof are necessary
in order to deal with the fact that the inequality~\eqref{eq:def-IC} only holds for convex functions).

\begin{theorem}\label{prop:IC-moments}
	Let $X$ be a symmetric random vector with values in $\RR^n$ which moments grow $\alpha$-regularly. Suppose moreover that $X$ satisfies convex $\IC$. Then for  every norm $\|\cdot \|$ on $\RR^n$ and  every $p\ge 2$  we have
	\begin{equation*}
		\Bigl( \EE \bigl|\|X\|-E\|X\|\bigr|^p \Bigr)^{1/p} \leq   C \alpha\beta\sigma_{\|\cdot\|, X}(p) ,
	\end{equation*}
	where $C$ is a universal constant (one can take $C= 4\sqrt{2}e < 16$).
\end{theorem}

Immediately we obtain the following corollary in the spirit of the results from \cite{MR2276533,
 MR3150710, lat-strz-2, MR3503729}.  Similar inequalities for Rademacher sums with the emphasis on exact values of constants have also been studied by Oleszkiewicz (see \cite[Theorem~2.1]{MR3273447}).

\begin{corollary}\label{cor:weak-strong}
	Let $X$ be a symmetric random vector with values in $\RR^n$ and with independent coordinates which have log-concave tails.
 Then for  every norm $\|\cdot \|$ on $\RR^n$ and  every $p\ge 2$  we have
		\begin{equation}\label{eq:comp_mom2bis}
		\bigl( \EE \|X\|^p \bigr)^{1/p} \le  \EE \|X\| +  D\sigma_{\|\cdot\|, X}(p),
	\end{equation}
	where $D$ is a universal constant (one can take $D = 6720\sqrt{2}e^2 < 70223$). 
\end{corollary}

Note that each of the terms on the right-hand side of~\eqref{eq:comp_mom2bis} is, up to a constant, dominated by the left-hand side of~\eqref{eq:comp_mom2bis}, so \eqref{eq:comp_mom2bis} yields  the comparison of weak and strong moments of the norms of $X$.

Note also that the constant standing at $\EE \|X\|$ is equal to $1$. If we only assume that the coordinates of $X$ are independent and their moments grow $\alpha$-regularly, then~\eqref{eq:comp_mom2bis} does not always hold (the example here is a vector with independent coordinates distributed like in~\eqref{eq:example-intro}; see Section~\ref{sec:example} for details), although by~\cite[Theorem~1.1]{lat-strz-2} it holds if we allow the constant at $\EE \|X\|$ to be greater than $1$ and to depend on $\alpha$. Hence Corollary \ref{cor:weak-strong} and example ~\eqref{eq:example-intro} partially answer  the following question raised in~\cite{lat-strz-2}: ``For which vectors does the comparison of weak and strong moments hold with constant $1$ at the first strong moment?''

The organization of the paper is the following. In Section~\ref{sec:proof-main} and~\ref{sec:comparison} we present the proofs of Theorem~\ref{thm:main} and Proposition~\ref{prop:IC-moments} respectively. In Section~\ref{sec:example} we discuss example~\eqref{eq:example-intro} in details.

\section{Proof of Theorem~\ref{thm:main}}\label{sec:proof-main}

Our approach is based on a characterization -- provided by  Gozlan, Roberto, Samson, Shu, and Tetali in~\cite{grsst15} -- of measures on the real line which satisfy a weak transport-entropy inequality. We emphasize that our optimal cost functions need not be quadratic near the origin, therefore we cannot apply their characterization as is, but have to first fine-tune the cost functions a bit. We shall also need the following simple lemma.

\begin{lemma}\label{lem:x^2}
If $X$ is a symmetric random variable and $\EE X^2 = \beta_1^{-2}$, then 
\[
\Lambda_X^*(x/\beta_1) \le x^2\quad \text{for } |x|\le 1.\]
\end{lemma}

\begin{proof}
Since $X$ is symmetric, we have
\begin{equation*}
	 \EE e^{tX} = 1 + \sum_{k=1}^\infty  \frac{\|X\|_{2k}^{2k} t^{2k} }{(2k)!} \ge 1+   \sum_{k=1}^\infty    \frac{\|X\|_2^{2k} t^{2k} }{(2k)!} =  1+   \sum_{k=1}^\infty    \frac{ \beta_1^{-2k}t^{2k} }{(2k)!}
	  =  \cosh(\beta_1^{-1}|t|).
\end{equation*}
Moreover, $\calL \bigl(\ln \cosh(\cdot)\bigr)( |u|)\le |u|^2$ for $|u|\le 1$ (see for example the proof of~ ~\cite[Proposition~3.3]{MR2449135}). Therefore
\begin{equation*}
	\Lambda_X^*(x/\beta_1) = \mathcal{L}(\Lambda_X(\beta_1\cdot)) (x) \le \mathcal{L}(\ln \cosh(\cdot))(x) \le x^2 \qquad \mbox{for } |x|\le 1.\qedhere
\end{equation*}
\end{proof}

Throughout the proof $g^{-1}$ stands for the generalized inverse of a function $g$  defined as
\begin{equation*}
	g^{-1}(y):= \inf \{ x: g(x) \ge y \}.
\end{equation*}

\begin{proof}[Proof of Theorem~\ref{thm:main}] Note that $N(0)=0$ and the function $N$ is non-decreasing. First we tweak the assumptions and change the assertion to a more straightforward one.

\krok{Step 1 (first reduction).} We claim that it suffices to prove the assertion for random variables for which the function $N$ is strictly increasing on the set where it is finite (or, in other words, $N(t)=0$ only for $t=0$). Indeed, suppose we have done this and let now $X$ be any random variable satisfying the assumptions of the theorem. Let $X_{\eps}$ be a symmetric random variable such that $\PP(|X_{\eps}|\geq t) = \exp(-N_{\eps}(t))$, where $N_{\eps}(t) =N(t)\lor \eps t$. If $X$ and $X_{\eps}$ are represented in the standard way by the inverses of their CDFs on the probability space $(0,1)$, then $|X_{\eps}|\leq |X|$ a.s. (and also $X_{\eps}\to X$ a.s. as $\eps\to 0^+$). Hence $\Lambda_{X_{\eps}}\leq \Lambda_{X}$ and therefore also $\Lambda^*_{X_{\eps}}\geq \Lambda^*_{X}$. 

The theorem applied to the random variable $X_{\eps}$ and the above inequality imply that the pair $(X_{\eps}, \Lambda^*_{X}(\cdot/\beta))$ satisfies the convex ICI. Taking $\eps\to 0^+$ we get the assertion for $X$ (in the second integral we just use the fact that the test function $f$ is bounded from below and thus $e^{-f}$ is bounded from above; for the first integral it suffices to prove the convergence of integrals on any interval $[-M,M]$, and on such an interval we have $f\infconv \Lambda^*_{X} (x/\beta) \leq f(x) +  \Lambda^*_{X} (0) =f(x)$, and thus $\exp(\max_{[-M,M]} f)$ is a good majorant).

\krok{Step 2 (second reduction).} We claim that it suffices to prove the assertion for random variables such that $\Lambda_X<\infty$. Indeed, suppose we have done this and let $X$ be any random variable satisfying the assumptions of the theorem. Let $N_{\eps}(t) =N(t)\lor \eps^2 t^2$ and let $X_{\eps}$ be a symmetric random variable such that $\PP(|X_{\eps}|\geq t) = \exp(-N_{\eps}(t))$. Then, similarly as in Step 1., $\Lambda_{X_{\eps}}\leq \Lambda_Y <\infty $, where $Y$ is symmetric and $\PP(|Y|\geq t) = \exp(- \eps^2 t^2)$. Thus we can apply the proposition to $X_\eps$ and we continue as in Step 1.

\krok{Step 3 (scaling).}  Due to the scaling properties of the Legendre transform, we can assume that $\EE X^2 = \beta_1^{-2}$, where $\beta_1:=  2e$ (the case where $X\equiv 0$ is trivial). Note that then, by Markov's inequality, $e^{-N(1/2)} = \PP (|X|\ge \frac 12) \le 4 \EE X^2 = e^{-2}$, so 
\begin{equation}\label{N(1/2)}
N(1/2) \ge 2.
\end{equation}

\krok{Step 4 (reformulation).} For $x\in \RR$ let
\begin{equation*}
\varphi(x):= \bigl(x^2 \indic{\{|x|<1\}} +(2|x|-1)\indic{\{|x|\ge1\}} \bigr) \lor \Lambda_X^*(x/(2{\beta_1})).
\end{equation*}  We claim that there exists a universal constant $\widetilde{b}\leq 1/420$, such that the pair $(X, \varphi(\tilde{b} \cdot))$ satisfies the convex infimum convolution inequality.  Of course the assertion follows immediately from that.

Note that $\varphi$ is convex, increasing on $[0,\infty)$ (because $\Lambda_X^*(\cdot /(2\beta_1))$ is convex and symmetric and thus non-decreasing on $[0,\infty)$). Crucially, $\varphi(x)=x^2$ for $x\in [0,1]$ (by Lemma~\ref{lem:x^2}), so the cost function $\varphi$ is quadratic near zero. Moreover, by Lemma \ref{lem:x^2}, $\varphi^{-1}(3)= 2$.

 Let $U=F^{-1} \circ F_{\nu}$, where $F$, $F_{\nu}$ are the distribution functions of $X$ and the symmetric exponential measure $\nu$ on  $\RR$, respectively. By~\cite[Theorem~1.1]{grsst15} we know that if there exists $b>0$ such that for every $x,y \in \RR$ we have
	\begin{equation}\label{cond_v1}
		\big|U(x)-U(y) \big| \le \frac{1}b \varphi^{-1} \bigl(1+ |x-y| \bigr),
	\end{equation}
then the pair $(X, \varphi(\widetilde{b}\cdot))$, where $\widetilde{b} = \frac{b}{210\varphi^{-1}(2+1^2)} = \frac{b}{420} $, satisfies the convex ICI. We will show that~\eqref{cond_v1} holds with $b=1$.

\krok{Step 5 (further reformulation).} Let $a=\inf\{t>0 : N(t) =\infty\}$. We have three possibilities (recall that $N$ is left-continuous):
\begin{itemize}
\item $a=\infty$. Then $N$ is continuous, increasing, and transforms $[0,\infty]$ onto $[0,\infty]$. Also, $F$ is increasing and therefore $F^{-1}$ is the usual inverse of $F$.
\item $a<\infty$ and $N(a)<\infty$. Then $X$ has an atom at $a$. Moreover, $N(a) = \lim_{t\to a^-}N(t)$.
\item $a<\infty$ and $N(a)=\infty = \lim_{t\to a^-}N(t)$.
\end{itemize}
Of course, in the first case one can extend $N$ by putting $N(a)=\infty$, so that all formulas below make sense.

Note that
\begin{equation*}
F(t) =
\begin{cases}
 \frac{1}{2}\exp(-N(|t|)) & \text{if } t<0,\\
 1-\frac{1}{2}\exp(-N_+(t)) & \text{if } t\geq0,\\
\end{cases}
\end{equation*}
where $N_+(t)$ denotes the right-sided limit of $N$ at $t$ (which is different from $N(t)$ only if $t=a$ and $X$ has an atom at $a$). Hence, $F$ is continuous on the interval $(-a,a)$, the image of $(-a,a)$ under $F$ is the interval $\big(\frac{1}{2}\exp(-N(a)), 1-\frac{1}{2}\exp(-N(a))\big)$, and we have $F(-a) = \frac{1}{2}\exp(-N(a))$ and $F(a) = 1$. Since the image of $\RR$ under $U$ is equal to the image of $(0,1)$ under $F^{-1}$, we conclude that $U(\RR) = (-a,a)$ if $N(a)=\infty$ and $U(\RR) = [-a,a]$ if $N(a)<\infty$. Denote $A:=U(\RR)$.

When $N(a) < \infty$, it suffices to check condition~\eqref{cond_v1} for $x,y\in [-N(a),N(a)] $  (otherwise one can change $x$, $y$ and decrease the right-hand side while not changing the value of the left-hand side of \eqref{cond_v1}). For $x\in  [-N(a),N(a)]$ we can write $U^{-1}(x)= N(|x|)\sgn(x)$ and $U^{-1}(x)\in\RR$. When $N(a) = \infty$, $U$ is a bijection (on its image), so we can obviously write again  $U^{-1}(x)= N(|x|)\sgn(x)$ for any $x \in \RR$.

Therefore, in order to verify~\eqref{cond_v1}  we need to check that
	\begin{equation}\label{cond_v1.5}
	|x-y |\le \varphi^{-1} \bigl( 1+\big|N(|x|)\sgn(x)- N(|y|)\sgn(y) \big| \bigr) \qquad \text{for } x, y\in A .
	\end{equation}

	Since we consider the case when $\Lambda_X(t)$ is finite for every $t\in \RR$, the Chernoff inequality applies, so for $t\ge \EE X=0$ we have 
	\begin{equation*}
		\frac 12 e^{-N(t)} = \PP( X\ge t) \le  e^{-\Lambda_X^*(t)},
	\end{equation*}
	so 
	\begin{equation}\label{Chern}
		N(t) \ge \Lambda_X^*(t) -\ln2.
	\end{equation}

 Note that $\varphi(|x-y|)<\infty$ for $x,y\in A$, since $\varphi(|x-y|)=\infty$ would imply $\Lambda^*_X (|x-y|/(2\beta_1))=\infty$, and hence  $\Lambda^*_X (|x-y|/2)=\infty$,  and -- by \eqref{Chern} -- also $N(|x-y|/2)=\infty$, but for $x,y\in A$ we have $|x-y|/2\in[0,a)$ when $N(a)=\infty$ or $|x-y|/2\in[0,a]$ when $N(a) < \infty$ and in either case $N(|x-y|/2)$ is finite. Therefore for every $x,y\in A$ we have $\varphi(|x-y|)<\infty$. Since $\varphi^{-1}(\varphi(z))=z$ for $z$ such that $\varphi(z)<\infty$ (because $\varphi$ is then continuous and increasing on $[0,z]$), the condition~\eqref{cond_v1.5} is implied by
	\begin{equation}\label{cond_v2}
	\varphi\bigl( |x-y | \bigr) \le 1+ \big|N(|x|)\sgn x- N(|y|)\sgn y \big| \qquad \text{for } x,y\in A.
	\end{equation}
In the next step we check that this is indeed satisfied.
	
\krok{Step 6 (checking the condition).} Let $x_0 = \inf\{ x\geq 1 : 2x-1=\Lambda_X^*(\frac{x}{ 2\beta_1})\}$ (if $x_0=\infty$ we simply  do not have to consider Case 2 below).  We consider three cases. We repeatedly use the fact that $uN(t)\ge N(ut)$ for $u\le 1$, $t\ge 0$, which follows by the convexity of $N$ and the property $N(0)=0$.
	
	\textbf{Case 1.} $|x-y|\le 1$.  Then $\varphi\bigl( |x-y | \bigr) =  ( x-y  ) ^2 \le 1$, so \eqref{cond_v2} is trivially satisfied.
	
	\textbf{Case 2.}  $|x-y|\ge x_0$. Then $\varphi\bigl( |x-y | \bigr) = \Lambda_X^*(\frac{1}{2\beta_1}|x-y|)\leq \Lambda^*_X(|x-y|/2)$. Inequality \eqref{Chern} implies that in order to prove \eqref{cond_v2} it suffices to show that if $x$, $y$ are of the same sign, say $x, y  \ge 0$, then $N\bigl(|x-y|/2) \le |N(x)-N(y)| $ and if $x,y$ have different signs, we have $N\bigl(\bigl(|x|+|y|\bigr)/2\bigr) \le N(|x|) + N(|y|)$.
	
By the convexity of $N$, for $s, t \geq 0$ we have 
	\begin{equation*}
		 N\bigl((s+t)/2\bigr) \le \frac 12 N(s)+ \frac 12 N(t) \le N(s)+N(t)
	\end{equation*}
and
	\begin{equation*}
		N(s/2)+N(t)\le N(s)+N(t)
		 \le \frac{s}{s+t}N(s+t)+ \frac{t}{s+t}N(s+t)  = N(s+t).
	\end{equation*}
	 This finishes the proof of~\eqref{cond_v2} in Case 2.
	
	\textbf{Case 3.} $1\le |x-y| \le x_0$. Then $\varphi\bigl(|x-y|\bigr) = 2|x-y|-1$. Consider two sub-cases:
	\begin{itemize}
	\item[(i)] $x, y$ have different signs. Without loss of generality we may assume  $x\geq |y|\geq 0\geq y$. Thus in order to obtain~\eqref{cond_v2} it suffices to show that $N(x)\ge 2x+2|y|$. Note that $1\le x + |y|\le 2x$, so $x\ge \frac{1}2$. Thus
	\begin{equation*}
		N(x)\ge N(1/2)2x \mathop{\ge}^{\eqref{N(1/2)}} 4x\ge 2x +2|y|,
	\end{equation*}
which finishes the proof in case~(i).

\item[(ii)] $x, y$ have the same sign. Without loss of generality we may assume $x\geq y\geq 0.$ Thus it suffices to show that $2(x-y) \le N(x)-N(y)$. Note that due to the assumption of Case~3 we have $x\ge x-y \ge  1 \ge \frac 12$, so by the convexity of $N$ we have 
	\begin{equation*}
		\frac{N(x)-N(y)}{x-y} \ge \frac{N(\frac 12)-N(0)}{\frac12-0} \mathop{\ge}^{\eqref{N(1/2)}} 4 \ge 2
	\end{equation*}
	This ends the examination of case~(ii) and the proof of the theorem. \qedhere
	\end{itemize}
\end{proof}

\section{Comparison of weak and strong moments}\label{sec:comparison}
	
The goal of this section is to establish the comparison of weak and strong moments with respect to any norm $\|\cdot\|$ for random vectors $X$ with independent coordinates having log-concave tails (Corollary \ref{cor:weak-strong}). In view of Theorem \ref{thm:main} and Remark \ref{rem:tails-reg}, it is enough to show Theorem \ref{prop:IC-moments}.  

Our proof of Theorem \ref{prop:IC-moments} comprises three steps: first we exploit  $\alpha$-regularity of moments of $X$ to control the size of its cumulant-generating function $\Lambda_X$, second we bound the infimum convolution of the optimal cost function with the convex test function being the norm $\|\cdot\|$ properly rescaled, and finally by the property convex $\IC$ we obtain exponential tail bounds which integrated out give the desired moment inequality.


We start with two lemmas corresponding to the first two steps described above and then we put everything together.

\begin{lemma}\label{lm:a-reg}
	Let $p \geq 2$ and suppose that the moments of a random vector $X$ in $\RR^n$ grow $\alpha$-regularly. If for a vector $u\in \RR^n$ we have $\|\langle u,X \rangle\|_p \leq 1$, then
	\[\Lambda_X((2e\alpha)^{-1}pu) \leq p.\]
\end{lemma}
\begin{proof}
Let $k_0$ be the smallest integer larger than $p$. If  $\alpha e \|\langle u, X\rangle \|_p \le 1/2$, then by $\alpha$-regularity we have
	\begin{align*}
		\Lambda_X(pu) 
		&\le \ln \Bigl(\sum_{k\ge 0} \frac{\EE |\langle pu, X \rangle |^k}{k!}  \Bigr)
		\le \ln \Bigl( \sum_{0\le k\le p} p^k\frac{\|\langle u, X\rangle \|_p^{k}}{k!}  + \sum_{k > p}  (\alpha  k )^k \frac{\|\langle u, X\rangle \|_p^{k}}{k!} \Bigr) \\
		&\leq \ln \Bigl( \sum_{0\le k\le p} \frac{p^k\|\langle u, X\rangle \|_p^{k}}{k!}  + \sum_{k > p}  \bigl(\alpha  e\|\langle u, X\rangle \|_p\bigr)^{k} \Bigr)\\
		&\leq \ln \Bigl( \sum_{0\le k\le p} \frac{p^k\|\langle u, X\rangle \|_p^{k}}{k!}  + 2 (\alpha e \|\langle u, X\rangle \|_p)^{k_0} \Bigr)\\
		&\mathop{\leq} \ln \Bigl( \sum_{0\le k\le p}\frac{ p^k\|\langle u, X\rangle \|_p^{k}}{k!}  +  \frac{(2\alpha e p \|\langle u, X\rangle \|_p )^{k_0}}{k_0!} \Bigr)	 \\	
		&\leq \ln \Bigl( \sum_{0\le k\le k_0}\frac{ (2\alpha ep\|\langle u, X\rangle \|_p)^{k}}{k!}  \Bigr)
		\le 2\alpha e p \|\langle u, X\rangle \|_p \leq p.
	\end{align*}
Replace $u$ with $(2e\alpha)^{-1}u$ to get the assertion.
\end{proof}

\begin{lemma}\label{lm:infconv-lowerbound}
Let  $\|\cdot \|$ be a norm on $\RR^n$ and let $X$ be a random vector with values in $\RR^n$ and moments growing $\alpha$-regularly. For $\beta > 0$, $p \geq 2$, and $x \in \RR^n$ we have
\[
\big(\Lambda_X^*\left(\cdot/\beta\right)\square a\|\cdot\|\big) (x) \geq a\|x\| - p,
\]
where $a = p(2e\alpha\beta\sigma_{\|\cdot\|,X}(p))^{-1}$.
\end{lemma}
\begin{proof}
For $f(x) = a\|x\|$ with positive $a$ being arbitrary for now we bound the infimum convolution as follows
\begin{align*}
\big(\Lambda^*_X(\cdot/\beta)\square f\big) (x) &= \inf_y\sup_z \left\{ \beta^{-1}\langle y,z \rangle - \Lambda_X(z) + a \|x-y\| \right\} \\
&= \inf_y\sup_u \left\{ (2e\alpha\beta)^{-1}p\langle y,u \rangle - \Lambda_X((2e\alpha)^{-1}pu) + a \|x-y\| \right\} \\
&\geq \inf_y\sup_{u: \|\langle u,X \rangle\|_p \leq 1} \left\{ (2e\alpha\beta)^{-1}p\langle y,u \rangle - p + a \|x-y\| \right\},
\end{align*}
where in the last inequality we have used Lemma \ref{lm:a-reg}.  Choose $u = \sigma_{\|\cdot\|,X}(p)^{-1} v$ with $\|v\|_* \leq 1$ such that $\langle y,v \rangle = \|y\|$. Then clearly $\|\langle u,X \rangle\|_p \leq 1$ and thus
\[
\Lambda^*_X(\cdot/\beta)\square f (x) \geq \inf_y \left\{ (2e\alpha\beta\sigma_{\|\cdot\|,X}(p))^{-1}p\|y\| - p + a \|x-y\| \right\}.\]
If we now set $a = p(2e\alpha\beta\sigma_{\|\cdot\|,X}(p))^{-1}$, then by the triangle inequality we obtain the desired lower bound
\[
\big(\Lambda_X^*\left(\cdot/\beta\right)\square a\|\cdot\|\big) (x) \geq a\|x\| - p. \qedhere
\] 
\end{proof}

\begin{proof}[Proof of Theorem \ref{prop:IC-moments}]
Let $f(x) = a\|x\|$ with $a = p(2e\alpha\beta\sigma_{\|\cdot\|,X}(p))^{-1}$ as in Lemma \ref{lm:infconv-lowerbound}. Testing the property convex $\IC$ with $f$ and applying Lemma \ref{lm:infconv-lowerbound} yields 
\[
\EE e^{a\|X\|}\EE e^{-a\|X\|} \leq e^{p}.\]
By Jensen's inequality we obtain that both $\EE e^{a(\|X\|-\EE\|X\|)}$ and $\EE e^{a(-\|X\|+\EE\|X\|)}$ are bounded above by $e^p$. Thus Markov's inequality implies the tail bound
\[
\PP\left( a\big|\|X\|-\EE\|X\|\big| > t \right) \leq 2e^{-t}e^{p} \leq 2e^{-t/2}, \quad t \geq 2p.\]
Consequently,
\begin{align*}
a^p\EE\bigl|\|X\|-\EE\|X\|\bigr|^p &= \int_0^\infty pt^{p-1}\PP\left( a\big|\|X\|-\EE\|X\|\big| > t \right) dt \\
&\leq (2p)^p + 2\int_0^\infty pt^{p-1}e^{-t/2} dt = (2p)^p + 2\cdot 2^pp\Gamma(p) \\
&\leq  2(2p)^p.
\end{align*}
Plugging in the value of $a$ gives the result (we can take $C = 4\sqrt{2}e < 16$).
\end{proof}

\section{An example}\label{sec:example}

Let $X$ be a symmetric random variable defined by  $\PP(|X| >t )= T(t)$, where
\begin{equation}\label{eq:def-example}
 T(t) :=\indic{[0,2)} (t) + \sum_{k=1}^{\infty} e^{-2^k} \indic{[2^k, 2^{k+1})}(t), \quad t\geq 0,
\end{equation}
or, in other words, let $|X|$ have the distribution
\begin{equation*}
 (1-e^{-2})\delta_{2}+ \sum_{k=2}^{\infty} \bigl(e^{-2^{k-1}} - e^{-2^k} \bigr) \delta_{2^k} .
\end{equation*}
Let us first show that the moments of $X$ grow $3$-regularly, but $X$ does not satisfy $\IC$ for any $\beta<\infty$ (we also prove a slightly stronger statement later).

Let $Y$ be a symmetric exponential random variable. Then $Y$ has log-concave tails, so the moments of $Y$ grow $1$-regularly (see Remark \ref{rem:tails-reg}). Moreover, if $X$ and $Y$ are constructed in the standard way by the inverses of their CDFs on the probability space $(0,1)$, then
\begin{equation*}
|Y| \leq |X| \leq 2|Y|+2.
\end{equation*}
Therefore, for $p\geq q \geq 2$,
\begin{equation*}
\|X\|_p \leq 2\|Y\|_p + 2 \leq 2\frac{p}{q} \|Y\|_q  +2\leq 3 \frac{p}{q} \|X\|_q
\end{equation*}
(we used the fact that $|X|\geq 2$ in the last inequality). Thus the moments of $X$ grow $3$-regularly.

On the other hand, for every $h>0$ there exists $t>0$ such that
\begin{equation*}
\PP(|X|\geq t+h) = \PP(|X| \geq t).
\end{equation*}
Therefore by \cite[Theorem~1]{fmnw15} there does not exist a constant $C$ such that the pair $(X, \varphi(\cdot/C))$, where $\varphi(x) = \tfrac{1}{2}x^2 \indic{\{|x|\leq 1\}} + ( |x|-1/2)\indic{\{|x|>1\}}  $, satisfies the convex infimum convolution inequality. But, by  symmetry and the $3$-regularity of moments of $X$,
	\begin{align*}
		\Lambda_X(s) 
		&\le \ln \Bigl(1 + \sum_{k\ge 1} \frac{ s^{2k}\EE X^{2k}}{(2k)!}  \Bigr)
		\le \ln \Bigl(1 + \sum_{k\ge 1} \frac{ s^{2k} (3k)^{2k} \bigl(\EE X^{2}\bigr)^k}{(2k)!}  \Bigr) \\
		&\le \ln \Bigl(1 + \sum_{k\ge 1} s^{2k} (3e/2)^{2k} \bigl(\EE X^{2}\bigr)^k   \Bigr) = \ln \Bigl(1 + \sum_{k\ge 1} \bigl( 9e^2s^{2} \EE X^{2}/4\bigr)^k   \Bigr).
	\end{align*}
Thus for some $A, \eps >0$ we have  $\Lambda_X(s) \leq As^2$ for $|s|\leq \eps$ and $2A\eps^2\ge 1$. Hence
	\begin{align*}
		\Lambda_X^*(t) &  \geq \sup_{|s|\leq \eps} \{st - As^2\}   =
		\tfrac{1}{4A} t^2\indic{\{|t|\leq 2A\eps\}}+(\eps |t| - A\eps^2)\indic{\{|t|> 2A\eps\}}
		\\ &
		= 2A\eps^2 \varphi \bigl(t/(2A\eps)\bigr) \ge \varphi \bigl(t/(2A\eps)\bigr) . 
	\end{align*}
				We conclude that $X$ cannot satisfy $\IC$ for any $\beta$.

\begin{remark}
Let us also sketch an alternative approach. Take $a, c >0$,  $b\in\RR$, and denote $\varphi(x) = \min\{ x^2, |x|\}$,  $f(x)  = f_{a,b}(x) = a(x-b)_+$ for $x\in\RR$. One can check that
\begin{equation*}
\bigl( f \infconv \varphi(c\cdot) \bigr) (x) =
\begin{cases}
0 &\text{if } x\leq b,\\
c^2 (x-b)^2 & \text{if } b< x \leq b+1/c,\\
c(x-b) & \text{if } x > b+1/c,
\end{cases}
\end{equation*}
if $a>2c$. It is rather elementary but cumbersome to show that for any $c>0$ there exist $a>0$ and $b\in\RR$ such that~\eqref{eq:def-IC} is violated by the test function $f$. We omit the details.
\end{remark}

In fact, the above example shows that even a slightly stronger statement is  true: for vectors with independent coordinates with $\alpha$-regular growth of moments the comparison of weak and strong moments of norms does not hold with the constant $1$ at the first strong moment. More precisely, let $X_1, X_2, \ldots $ be independent random variables with distribution given by~\eqref{eq:def-example}.  We claim that there does \emph{not} exist any $K<\infty$ such that
\begin{equation}\label{eq:example-comparison}
\bigl(\EE \max_{i\leq n} |X_i|^p \bigr)^{1/p} \leq \EE \max_{i\leq n} |X_i| + K   \sup_{\|t\|_1\leq 1} \Bigl( \EE \bigl|\sum_{i=1}^n t_i X_i\bigr|^p \Bigr)^{1/p}
\end{equation}
holds for every $p\geq 2$ and $n\in\NN$ (note that we chose the $\ell^{\infty}$-norm as our norm). We shall estimate the three expressions appearing in~\eqref{eq:example-comparison}.

We have
\begin{equation} \label{eq:counterex-1}
\sup_{\|t\|_1\leq 1} \Bigl( \EE \bigl|\sum_{i=1}^n t_i X_i\bigr|^p \Bigr)^{1/p} \leq
\sup_{\|t\|_1\leq 1} \sum_{i=1}^n |t_i| \| X_i\|_p = \|X_1\|_p
\end{equation}
(this inequality is in fact an equality).  Since the moments  of $X_1$ grow $3$-regularly, the last term in~\eqref{eq:example-comparison} is bounded by $\widetilde{K}p$ for some $\widetilde{K}<\infty$.

To estimate the remaining two terms we need the following standard fact.

\begin{lemma}
For independent events $A_1, \ldots, A_n$,
\begin{equation*}
(1-e^{-1}) \Bigl( 1 \land \sum_{i=1}^n \PP(A_i) \Bigr) \leq \PP\Bigl(\bigcup_{i=1}^n A_i\Bigr) \leq  1 \land \sum_{i=1}^n \PP(A_i).
\end{equation*}
In particular, for i.i.d. non-negative random variables $Y_1, \ldots, Y_n$, 
\begin{equation*}
(1-e^{-1}) \int_0^{\infty} \Big[1 \land n \PP(Y_1 >t)\Big] dt \leq \EE \max_{i\leq n} Y_i \leq  \int_0^{\infty} \Big[1 \land n \PP(Y_1 >t)\Big] dt.
\end{equation*}
\end{lemma}

\begin{proof}
The upper bound is just the union bound. The lower bound follows from de Morgan's laws combined with independence and the inequalities $1-x \leq e^{-x}$ and $1-e^{- y} \geq (1-e^{-1})y$ for $x\in\RR$, $y\in[0,1]$.
\end{proof}

Fix $m\geq 2$ and let $e^{2^{m-1}} \leq n < e^{2^m}$. Then
\begin{equation*}
1\land n T(t)  =
\begin{cases}
1 & \text{if } 0< t< 2^m,\\
nT(t) &\text{if } t\geq 2^m.
\end{cases}
\end{equation*}
By the above lemma,
 \begin{align*}
\EE \max_{i\leq n} |X_i|
&\leq  \int_0^{2^m} dt + n \int_{2^m}^{\infty} T(t) dt = 2^m + n \sum_{j=m}^{\infty} e^{-2^j}(2^{j+1} - 2^j)\\
& = 2^m + n \sum_{j=m}^{\infty} e^{-2^j}2^j \leq 2^m + n e^{-2^m} 2^m \sum_{j=0}^{\infty} (2e^{-2^m})^j = 2^m +\frac{n e^{-2^m} 2^m}{1 - 2e^{-2^m}}.
 \end{align*}
 Set $\theta = \theta(m,n) = n e^{-2^m} \in [e^{-2^{m-1}}, 1)$. Then
\begin{equation}\label{eq:counterex-2}
\EE \max_{i\leq n} |X_i| \leq 2^m\Bigl(1 + \frac{\theta}{1 - 2e^{-2^m}} \Bigr).
 \end{equation}
 
 Similarly,
\begin{align*}
\EE \max_{i\leq n} |X_i|^p
&\geq (1-e^{-1}) \int_0^{\infty} 1\land T(t^{1/p}) dt \\
&= (1-e^{-1})\Big[\int_0^{2^{mp}}  dt  + n\int_{2^{mp}}^{\infty} T(t^{1/p})  dt \Bigr]\\
& =  (1-e^{-1})\Big[2^{mp}  +  n \sum_{j=m}^{\infty} e^{-2^j}\bigl(2^{(j+1)p} -2^{jp}\bigr) \Bigr].
 \end{align*}
 Hence
\begin{equation}
\label{eq:counterex-3}
\EE \max_{i\leq n} |X_i|^p >  (1-e^{-1}) n e^{-2^m}\bigl(2^{(m+1)p} -2^{mp}\bigr) = (1-e^{-1}) \theta 2^{mp} (2^p-1).
 \end{equation}
 
 Putting \eqref{eq:counterex-1}, \eqref{eq:counterex-2}, and \eqref{eq:counterex-3} together, we see that \eqref{eq:example-comparison}~would imply
 \begin{equation*}
 (1-e^{-1})^{1/p} \theta^{1/p} 2^{m} (2^p-1)^{1/p} \leq 2^m\Bigl(1 + \frac{\theta}{1 - 2e^{-2^m}} \Bigr) + \widetilde{K}p 
 \end{equation*}
for every $p\geq 2$, $m\geq 2$, and $\theta\in[e^{-2^{m-1}},1)$ of the form $ne^{-2^m}$, $n\in\NN$. Take $p=1/\theta$ and $\theta \sim 1/m$ to get
 \begin{equation*}
 (1-e^{-1})^{\theta} \theta^{\theta} (2^{1/\theta}-1)^{\theta} \leq 1 + \frac{\theta}{1 - 2e^{-2^m}} + \frac{\widetilde{K}}{2^m\theta }.
 \end{equation*}
Since $\theta\to 0$ and $2^m\theta \to \infty$ as $m\to\infty$ this inequality yields $2\leq 1$, which is a~contradiction. Hence inequality~\eqref{eq:example-comparison} cannot hold for all $p\geq 2$ and $n\in\NN$.

\section*{Acknowledgments}	

We thank Rados{\l}aw Adamczak and Rafa{\l} Lata{\l}a for posing questions which led to the results presented in this note.

\bibliographystyle{amsplain}
\bibliography{convinfconv-arxiv}

\end{document}